\documentclass[12pt]{amsart}

\usepackage[all]{xy}
\usepackage{amssymb}
\usepackage{amsmath}
\usepackage{stmaryrd}
\usepackage{mathabx}

\usepackage[OT2,T1]{fontenc}
\DeclareSymbolFont{cyrletters}{OT2}{wncyr}{m}{n}
\DeclareMathSymbol{\Sha}{\mathalpha}{cyrletters}{"58}

\newcommand{\segment}[2]{\subsection{#2}\label{#1}}
\newcommand{\ssegment}[2]{\subsubsection{#2}\label{#1}}

\theoremstyle{definition}
\newtheorem*{prop*}{Proposition}

\theoremstyle{definition}
\newtheorem*{thm*}{Theorem}

\theoremstyle{definition}
\newtheorem*{cor*}{Corollary}

\theoremstyle{definition} 
\newtheorem{ssprop}[subsubsection]{Proposition}

\theoremstyle{definition} 
\newtheorem{sprop}[subsection]{Proposition}
\newcommand{\sProposition}[2]{\begin{sprop} \label{#1} 
{#2} \end{sprop}}

\theoremstyle{definition} 
\newtheorem{sconj}[subsection]{Conjecture}

\theoremstyle{definition} 
\newtheorem{sthm}[subsection]{Theorem}

\theoremstyle{definition} 
\newtheorem{ssthm}[subsubsection]{Theorem}

\theoremstyle{definition} 
\newtheorem{sslm}[subsubsection]{Lemma}

\theoremstyle{definition} 
\newtheorem{slm}[subsection]{Lemma}

\theoremstyle{definition} 
\newtheorem{scor}[subsection]{Corollary}
\newcommand{\sCorollary}[2]{\begin{scor} \label{#1} 
{#2} \end{scor}}

\theoremstyle{definition} 
\newtheorem{ssconj}[subsubsection]{Conjecture}

\theoremstyle{definition}
\newtheorem*{conj*}{Conjecture}

\theoremstyle{definition} 
\newtheorem{sscond}[subsubsection]{Condition}

\theoremstyle{definition} 
\newtheorem{ssrmk}[subsubsection]{Remark}

\let\oldmarginpar\marginpar
\renewcommand\marginpar[1]{\-\oldmarginpar[\raggedleft\footnotesize #1]%
{\raggedright\footnotesize #1}}

\newcommand{\set}[2]{\big\{ #1 \; \big| \; #2 \big\} }

\newcommand{\Vect}{\operatorname{\bf{Vect}}}

\newcommand{\from}{\leftarrow}
\newcommand{\xto}{\xrightarrow}

\newcommand{\surj}{\twoheadrightarrow}

\newcommand{\Spec}{\operatorname{Spec}}

\newcommand{\Li}{\operatorname{Li}}

\newcommand{\Hom}{\operatorname{Hom}}
\newcommand{\Ext}{\operatorname{Ext}}
\newcommand{\Isom}{\operatorname{Isom}}

\newcommand{\Aut}{\mathrm{Aut}\,}

\newcommand{\Lie}{\operatorname{Lie}}

\newcommand{\m}[1]{\mathrm{#1}}
\newcommand{\fk}[1]{\mathfrak{#1}}

\newcommand{\bb}[1]{\mathbb{#1}}

\newcommand{\la}{\lambda}
\newcommand{\ka}{\kappa}
\newcommand{\si}{\sigma}
\newcommand{\Si}{\Sigma}
\newcommand{\ze}{\zeta}
\newcommand{\ga}{\gamma}

\newcommand{\be}{\beta}
\newcommand{\om}{\omega}

\newcommand{\ep}{\epsilon}
\newcommand{\de}{\delta}

\newcommand{\Gm}{{\mathbb{G}_m}}
\newcommand{\Qp}{{\QQ_p}}
\newcommand{\Zp}{{\ZZ_p}}

\newcommand{\Cc}{\mathcal{C}}

\newcommand{\ZZ}{\bb{Z}}

\newcommand{\nN}{\fk{n}}

\newcommand{\NN}{\bb{N}}
\newcommand{\QQ}{\bb{Q}}
\newcommand{\RR}{\bb{R}}
\newcommand{\PP}{\bb{P}}

\newcommand{\Uu}{\mathcal{U}}
\newcommand{\uU}{\mathfrak{u}}

\newcommand{\Bb}{\mathcal{B}}

\renewcommand{\AA}{\bb{A}}

\newcommand{\Oo}{\mathcal{O}}

\newcommand{\Aa}{\mathcal{A}}
\newcommand{\Ee}{\mathcal{E}}
\newcommand{\Pp}{\mathcal{P}}

\newcommand{\areq}{\ar@{=}}
\newcommand{\suphook}{\ar@{^(->}}
\newcommand{\subhook}{\ar@{_(->}}

\newcommand{\smses}[6]
{
\[
\xymatrix{
1 \ar[r] &
#1 \ar[r]_-{#2} &
#3 \ar[r]_-{#4} &
#5 \ar[r] \ar@/_1.5pc/[l]_-{#6} &
1
}
\]
}

\newcommand{\inj}{\hookrightarrow}

\newcommand{\thrpl}{\PP^1 \setminus \{0,1,\infty\}}

\newcommand{\dR}{{\rm {dR}}}

\newcommand{\reg}{\operatorname{reg}}

\newcommand{\ab}{\mathrm{ab}}

\newcommand{\PL}{\m{PL}}
\newcommand{\ev}{\fk{ev}}

\newcommand{\per}{\operatorname{per}}

\newcommand{\MT}{\operatorname{\mathbf{MT}}}

\newcommand{\logu}{\log^\uU}
\newcommand{\Liu}{\Li^\uU}
\newcommand{\zeu}{\ze^\uU}

\newcommand{\un}{\m{un}}

\newcommand{\wt}{\operatorname{wt}}

\title
[The polylog quotient and the Goncharov quotient II]
{The polylog quotient and the Goncharov quotient in computational Chabauty-Kim theory II}

\author{Ishai Dan-Cohen and David Corwin}

\thanks{
D.C. was supported by through NSF RTG grant 1646385, through NSF grants DMS-1069236 and DMS-1601946 (to Bjorn Poonen), and
  by Simons Foundation grant \#402472 (to Bjorn Poonen).
I.D. was supported by ISF grant 87590021. Thanks are due to the Center for Advanced Studies in Mathematics for its support during the first author's visits to Ben Gurion University}

\date{\today}

\begin{document}

\maketitle

\raggedbottom
\SelectTips{cm}{11}

\begin{abstract} 
This is the second installment in a multi-part series starting with \cite{CorwinDCI}. Building on Dan-Cohen--Wewers \cite{mtmue},  Dan-Cohen \cite{mtmueii}, and  Brown \cite{BrownUnit}, we push the computational boundary of our explicit motivic version of Kim's method in the case of the thrice punctured line over an open subscheme of $\Spec \ZZ$. To do so, we develop a refined version of the algorithm of \cite{mtmue} tailored specifically to this case. We also commit ourselves fully to working with the polylogarithmic quotient. This allows us to restrict our calculus with motivic iterated integrals to the so-called depth-one part of the mixed Tate Galois group studied extensively by Goncharov. An application was given in \cite{CorwinDCI} where we verified Kim's conjecture in an interesting new case. 
\end{abstract}

\section{Introduction}

Work by Dan-Cohen--Wewers \cite{CKtwo, mtmue} and by Dan-Cohen \cite{mtmueii} produced an algorithm, based on the Chabauty--Kim method, for computing the integral points of $\thrpl$ over open integer schemes whose halting was, and remains, conditional on deep conjectures by Kim and by Goncharov, among others.\footnote{
Practical (and unconditional) methods for solving the $S$-unit equation predate this work, and can be found, for instance, in de Weger \cite{deWeger} who uses the theory of logarithmic forms of Baker--W\"ustholz \cite{BakerWuestholz} (see also Everste--Gy\H{o}ry \cite{EversteGyory} for a general discussion). A more recent approach, due to K\"anel--Matchke \cite{KaenelMatschke} is based on the Shimura-Taniyama conjecture. Our primary purpose here is not to compete with these other methods, but rather, to develop Kim's theory in a special case, to explore its interaction with the theory of mixed Tate motives and motivic iterated integrals, and to provide new numerical evidence for Kim's conjecture.
} This algorithm restricted attention to the polylogarithmic quotient of the unipotent fundamental group $\pi_1^\un(X)$, but minimized its reliance on the polylogarithmic quotient with a view towards eventually constructing an algorithm for the full unipotent fundamental group. The resulting algorithm (to quote an anonymous referee) provides proof of concept. In terms of practical applications, however, it is quite unwieldy. 

Francis Brown \cite{BrownUnit} introduced various new techniques which he was able to use to construct functions for open subschemes of $\Spec \ZZ$ in several examples. His techniques capitalize on the relative simplicity of $\Spec \ZZ$ and of the polylogarithmic quotient, and inspired us to attempt to construct a simpler algorithm, and to continue to push the computational boundary, for open subschemes of $\Spec \ZZ$. Brown also suggested (in private communication) that we might be able to replace the full mixed Tate Galois group with a quotient tailored specifically to the polylogarithmic quotient of the unipotent fundamental group.

In part one of this series \cite{CorwinDCI} (currently expected to have three parts), we worked out an example which goes beyond the examples of \cite{CKtwo, mtmue, mtmueii}. This example displayed several new and interesting phenomena. It also helped us refine our methods, and led the way to the work presented here. 

In this second installment, we present an algorithm based on the example considered in part I on the one hand, and based on Goncharov's investigations of the \emph{depth-1} quotient \cite{GonMot} on the other.\footnote{
The remainder of the introduction presupposes some familiarity with the surrounding literature (on mixed Tate motives, the unipotent fundamental group, motivic iterated integrals, $p$-adic polylogarithms, Kim's method, and Kim's conjecture) as well as passing familiarity with the previously constructed algorithm of \cite{mtmue, mtmueii}, and focuses on comparing old with new. The reader who is unfamiliar with this material may wish to skip forward to section 2 (where some of this background material is reviewed), or to go back to \cite{mtmue} for a general introduction.
} Let $K$ be a number field and let $Z$ be either $\Spec K$ (a ``number scheme'') or open in $ \Spec \Oo_K$  (an ``open integer scheme''). Associated to the \emph{Goncharov quotient}, as we will call it (segment \ref{table}),  is a Hopf subalgebra $A^G(Z)$ of the mixed Tate Hopf algebra $A(Z)$ of framed mixed Tate motives over $Z$. This subalgebra already contains the extension spaces
\[
K^{(n)}_{2n-1}(Z) = \Ext^1_Z(\QQ(0), \QQ(n)) \subset A_n(Z),
\]
and was studied extensively by Goncharov in connection with Zagier's conjecture relating special values of zeta functions of number fields to polylogarithms. This eventually led to his \textit{depth filtration conjecture} \cite{GonGal}; its depth-1 part says that $A^G(\Spec K)$ is generated by Goncharov's (unipotent) motivic polylogarithms $\Liu_i(x)$ (for $i \ge 1$ and $x$ a $K$-point of $\thrpl$). Although Zagier's conjecture, if stated in terms of values of zeta functions, is tautological for $\QQ$, the \textit{depth-1 conjecture} is still open even in this case. For our purposes we must refine this conjecture somewhat with respect to ramification. The resulting \textit{integral depth-1 conjecture} (\S\ref{depth}) for open subschemes of $\Spec \ZZ$ is based on a study of the \textit{half-weight-2} part from \cite{CKtwo}, and is of interest, we hope, in its own right. Indeed, we believe that an investigation of this conjecture and an attempt to generalize it to higher number fields (even in half-weight 2, where the case $Z = \Spec K$ is known and where the difference between $A(Z)$ and $A^G(Z)$ is not yet visible) may lead to near-term progress towards Goncharov's conjecture. 

Focusing attention on the \emph{Goncharov quotient} presents several simultaneous advantages. Most obviously, the dimensions of the graded pieces $A^G_n(Z)$ are far smaller than those of $A(Z)$. Correspondingly, if our integral depth-1 conjecture holds, our search through the vast collection of motivic iterated integrals for candidate basis elements may be limited to the comparatively small set of $n$-logarithms. In turn, working with $n$-logarithms allows us to avoid the complex combinatorics of the Goncharov coproduct formula \cite[Theorem 1.2]{GonGal}. Less obvious but perhaps equally important is that working with the Goncharov and polylogarithmic quotients allows us to make the geometry of the ``geometric algorithm'' (which computes the scheme theoretic image of the universal cocycle-evaluation map) fully explicit: we obtain a homomorphism of polynomial $\QQ$-algebras given by an explicit family of polynomials about which we learned from Brown. As a result, our new algorithm is far simpler and more efficient than the algorithm of \cite{mtmue}. It is therefore reasonable and worth-while to give a more explicit construction, with the promise of actual Sage code and ensuing numerical results (beyond those of \cite{CorwinDCI}) in the near future. 

While the Goncharov quotient holds much promise, it also presented us with a challenge. Unlike the full mixed Tate Galois group, the Goncharov quotient is not free, and so its coordinate ring $A^G(Z)$ is not a shuffle algebra. This threatened to send us looking through Goncharov's intricate analysis for information (actual or conjectural) about higher extension groups in the corresponding category of ``Goncharov motives''. Fortunately, as we discovered, by remembering the inclusion
\[
A^G(Z) \subset A(Z)
\]
and working sometimes in $A^G(Z)$ and sometimes in $A(Z)$, we can avoid direct confrontation with the structure of $A^G(Z)$.

A key observation for our work is that replacing the full mixed Tate Galois group of $Z$ by its Goncharov quotient does not shrink the  $p$-adic analytic loci $X(\Zp)_n \subset X(\Zp)$ that our algorithm computes. Actually, our loci are potentially \textit{larger} than those that intervene in Kim's conjecture \cite{nabsd}. One difference comes, of course, from considering only the polylogarithmic quotient of the unipotent fundamental group in place of the full unipotent fundamental group. The belief that we should nevertheless have equality 
\begin{align*}
\tag{*}
X(Z) = X(\Zp)_n
&&
\mbox{for } n \mbox{ sufficiently large}
\end{align*}
has travelled down a somewhat bumpy road. Kim showed in \cite{kimi} that the loci $X(\Zp)_n$ associated to the polylogarithmic quotient are finite for sufficiently large $n$ and in \cite{KimTangential} that the same holds for open integer schemes associated to totally real number fields. Experts expressed the hope (if only tentatively and quietly) that the polylogarithmic quotient should be \textit{big enough} for Kim's conjecture. However, upon completing our realization of the case $\ZZ[1/3]$ in \cite{CorwinDCI}, we discovered by computing numerically that $-1 \in X(\Zp)_4$ in that case, and subsequently showed more generally that for $Z = \Spec \ZZ[1/q]$ (for any prime $q \neq p$), we have
\begin{align*}
-1 \in X(\Zp)_n
&&
\mbox{for all } n.
\end{align*}
This meant that (*) was in fact false as stated, since for $q \neq 2$, $X(\ZZ[1/q]) = \emptyset$ (so, in particular, does \emph{not} contain $-1$). Nevertheless, we view this bump in the road as reflecting a particular interaction of the polylogarithmic quotient with roots of unity, or, at the very least, as a result of its symmetry-breaking nature. We may thus symmetrize our loci with respect to the $S_3$ action and, thus modified, expect property (*) to hold after all.

There is also a second, less obvious difference between Kim's loci and ours. This too intervened already in \cite{mtmueii}. While Kim's version may be defined in terms of the scheme-theoretic image of a $p$-adic realization map between nonabelian cohomology varieties over $\Qp$, our algorithm computes the scheme theoretic image of the rationally-defined universal cocycle evaluation map, and only then pulls back the result to $\Qp$. The result is again a possible enlargement of the resulting loci. However, as discussed in loc. cit., if the $p$-adic period conjecture holds, then there is no such enlargement after all. Period conjecture or no period conjecture, the final result is that our algorithm may be used to \textit{verify} Kim's conjecture, but can have little bearing on any attempt to falsify the conjecture. 

Because of its relative simplicity, we find it appropriate to describe the new algorithm in a somewhat more informal style than that of loc. cit., which we hope will make the present article more approachable. A concrete case (``$\ZZ[1/3]$ in depth $4$''), which provides new numerical evidence for Kim's conjecture, is worked out in \cite{CorwinDCI}. Let us point out, however, that \textit{that} work should not be viewed as merely executing the algorithm constructed in \textit{this} work since the history is rather reversed: as the numbering suggests, the example worked out in \cite{CorwinDCI} came first, and many of the ideas that went into constructing the present algorithm first developed in working out the example. 

\subsection*{Acknowledgements}
We would like to thank Bjorn Poonen and Herbert Gangl for their interest and encouragement. We would like to thank the referee for a careful reading and many helpful comments. 

\section{Conjectures and theorems}

In this section, which doubles as a long introduction, we give precise statements of the conjectures on which our algorithm depends for halting. We also announce the success of our algorithm in computing the set of integral points and the conditional halting in a complementary pair of theorems (\ref{pointcountingthm}) similar to the main theorems of \cite{mtmueii}. As promised, we begin by summarizing background material; details may be found in \cite{mtmue, mtmueii, CorwinDCI} and the references given there.

\segment{mtm}{Mixed Tate motives and mixed Tate Galois group}
Fix an open subscheme $Z \subset \Spec \ZZ$ and a prime $p \in Z$. Write 
\[
S := \Spec \ZZ \setminus Z = \{q_1, \dots, q_s\}.
\]
We let $K^{(n)}_m(Z)$ denote the $n$-eigenspace for the Adams operations on the rational Quillen K-group $K_m(Z)\otimes \QQ$. Historically (due in large part to the work of Borel \cite{Borel53, Borel77}), it was known that 
\[
K^{(n)}_{2n-1}(Z)=
\left\{
\begin{matrix}
\QQ \langle \logu (q_1), \dots, \logu(q_s) \rangle
& 
n=1,
\\
\QQ \zeu(n) 
&
n \mbox{ odd } \ge 3,
\end{matrix}
\right.
\]
and that all other Adams pieces vanish. Here, $\logu(q_i)$ and $\zeu(n)$ are certain special elements. Moreover, there were certain naturally defined $p$-adic regulator maps
\[
\reg: K^{(n)}_{2n-1} \to \Qp
\]
and
\[
\reg(\logu(q_i)) = \log^p(q_i)
\]
\[
\reg(\zeu(n)) = \ze^p(n),
\]
the $p$-adic logarithm and zeta value respectively. 

The $p$-adic zeta values $\ze^p(n)$ ($n$ odd $\ge 3$) are known to be nonzero for $p$ regular. A small piece of the $p$-adic period conjecture asserts that the same holds for all primes $p$. See Examples 2.19 and Remark 2.20 of Furusho \cite{FurushoI} for a discussion of this and related conjectures.

\ssegment{mtmoverview}{}
The theory of mixed Tate motives puts the $K$-groups above inside a big Hopf algebra. To define it, let $\MT(Z)$ denote the category of mixed Tate motives unramified over $Z$. The category $\MT(Z)$ is $\QQ$-Tannakian; its simple objects up to isomorphism are precisely the tensor powers of $\QQ(1)$. De Rham cohomology gives rise to a fiber functor 
\[
\dR^*: \MT(Z) \to \Vect \QQ.
\]
We have 
\[
\Hom(\QQ(0), \QQ(0)) = \QQ,
\]
\[
\Ext^1(\QQ(0), \QQ(n)) = K^{(n)}_{2n-1}(Z)
\]
and all other Ext groups $\Ext^i(\QQ(0), \QQ(n))$ vanish. This means that $\MT(Z)$ is a \emph{mixed Tate category}, and the results that follow are formal consequences of the axioms; see, for instance, Appendix A of Goncharov \cite{GonGal}.

There is a semi-direct product decomposition 
\[
\pi_1^\m{MT}(Z) := \Aut^\otimes (\dR^*) = \pi_1^\un(Z) \rtimes \Gm
\]
in which 
$
\pi_1^\un(Z)
$
is free prounipotent with abelianization 
\[
\tag{*}
\pi_1^\un(Z)^\ab
=
\bigoplus _{n \ge 1}
\Ext^1(\QQ(0), \QQ(n))^\lor,
\]
and $A(Z)  := \Oo (\pi_1^\un(Z))$ is a \emph{shuffle algebra}: in terms of homogeneous generators
\begin{align*}
\tau_1, \dots, \tau_s
&&
\mbox{and}
&&
\si_3, \si_5, \si_7, \dots
\end{align*}
(with $\tau_i$ in degree $-1$ and $\si_i$ in degree $-i$)
 $A(Z)$ has vector space basis
\[
\set{f_w}
{w \mbox{ a word in } 
\tau_1, \dots,\tau_s, \si_3, \si_5, \dots }
\]
called a ``shuffle basis'' on which the ``shuffle'' product is given by
\[
f_{w'} f_{w''} =
\sum_{\mbox{shuffles } w \mbox{ of }w', w''}
f_w,
\]
 and the ``deconcatenation'' coproduct is given by
\[
\Delta f_w = 
\sum_{w = w'w''}
f_{w'} \otimes f_{w''}.
\] 
Equation (*) may be written dually as a family of exact sequences 
\[
\tag{$*^\lor$}
0 \to
\Ext^1(\QQ(0), \QQ(n)) \to A_n 
\xto{\Delta'}
\bigoplus_{i+j = n \; i,j \ge 1}
A_i \otimes A_j
\]
where $\Delta'$ is the \emph{reduced coproduct}:
\[
\Delta'(a) = \Delta(a) - (1 \otimes a + a\otimes 1).
\]

\segment{unfun}{Unipotent fundamental group and motivic polylogarithms}
Let $X = \PP^1_Z \setminus \{0,1,\infty\}$. We use the tangent vector $\vec{1_0}$ as base point. We define $\pi_1^\un(X)$ to be the fundamental group $\Aut^\otimes (\om_0)$ of the Tannakian category of unipotent connections on $X_\QQ$ at the fiber functor associated to $\vec{1_0}$. Similarly, we define the path torsor associated to a point $x\in X(Z)$ by
\[
\pi_1^\un(X,0,x) := \Isom^\otimes(\om_0, \om_x).
\]
According to Deligne \cite{Deligne89}, $\pi_1^\un(X,0,x)$ is canonically trivialized by a path ${_xp_0^\dR} \in \pi_1^\un(X,0,x)(\QQ)$ (see \cite[\S3.1.7]{mtmue}
 and the references given there). Moreover, $\pi_1^\un(X)$ is free prounipotent on two generators $e_0, e_1$ representing monodromy about $0$ and $1$, respectively. Consequently,
\[
A(X) := \Oo(\pi_1^\un(X))
\]
is a shuffle algebra on the shuffle basis $e_0, e_1$. Deligne--Goncharov \cite{DelGon} (with subsequent different approaches, for instance, by Levine \cite{LevineTMFG} and Dan-Cohen--Schlank \cite{RMPS}) construct a natural action of $\pi_1^\m{MT}(Z)$ on $\pi_1^\un(X,0,x)$.

Special elements $\Liu_n(x) \in A_n(Z)$ were constructed by Goncharov \cite{DelGon}; these were called \emph{motivic polylogarithms} in loc. cit. but are called \emph{unipotent motivic polylogarithms} by Francis Brown \cite{Brown} to distinguish them from his somewhat different notion. (In fact the $n$-logarithms (or \textit{polylogarithms}) $\Liu_n(x)$ form a special class among the more general \textit{multiple polylogarithms}, which, in turn, are a special class of (unipotent motivic) iterated integrals defined and studied in \cite{GonGal}). In the case of $Z \subset \Spec \ZZ$ (rather than more general integer schemes) these elements may be defined quite simply in terms of the action of $\pi_1^\m{MT}(Z)$ on $\pi_1^\un(X,0,x)$ and in terms of the special de Rham paths ${_xp_0^\dR}$, as we now recall.\footnote{
When $Z \subset \Spec \Oo_K$, $K$ a number field bigger than $\QQ$, the de Rham fiber functor is no longer $\QQ$-rational. Instead, one must work with the \textit{canonical} realization of the unipotent fundamental group, which thus loses its Tannakian interpretation in terms of connections. The ensuing construction of motivic polylogarithms, which is due to Goncharov, is more complicated. 
}

We abbreviate words in $e_0, e_1$ by writing words in $0$ and $1$, especially when these appear as subscripts. We let $o$ denote the orbit map associated to the de Rham path ${_xp_0^\dR}$ and we let $\tau$ denote the trivialization of the path torsor $\pi_1^\un(X,0,x)$ associated to the same path. Then $\Liu_n(x)$ is defined to be the composite map
\[
\tag{*}
\pi_1^\un(Z) 
\xto{o}
\pi_1^\un(X,0,x)
\underset{\tau}
{\xto{\sim}}
\pi_1^\un(X)
\xto{f_{10\cdots 0}}
\AA^1
\]
with $n-1$ zeroes below the `$f$'. (The intermediate map
\[
\ka(x):= \tau \circ o
\]
is a $\Gm$-equivariant 1-cocycle and will play an important role below.)
The unipotent logarithm is defined similarly by using the one-letter word `$0$' and satisfies
\[
 \Liu_1(x) = -\logu(1-x).
\]
The unipotent special zeta value $\zeu(n)$ is defined (using the tangential end-point $-\vec{1_1}$ in place of $x$) by
\[
\zeu(n) := \Liu_n(-\vec{1_1}).
\]
$p$-Adic integration gives rise to a ring homomorphism
\[
\per_p : A(Z) \to \Qp
\]
which extends the regulator maps and 
\[
\per_p(\Liu_n(x)) = \Li_n^{p}(x)
\]
is the $p$-adic $n$-logarithm of $x$.

\segment{new}{The polylog quotient and the Goncharov quotient}
The natural inclusion
\[
X \inj \Gm 
\]
gives rise to a map
\[
\pi_1^\un(X) \to
 \pi_1^\un(\Gm) =
  \QQ(1).
\]
Let $N$ be the kernel. The polylog quotient is defined as
\[
\pi^\PL(X):= \pi_1^\un(X)/[N,N].
\]
Let $\nN^\PL(X) := \Lie \pi^\PL(X)$. According to Proposition 16.13 of Deligne \cite{Deligne89},
 $\nN^\PL(X)$ is a sum of Tate motives. Hence $\pi_1^\un(Z)$ acts trivially on $\pi^\PL(X)$. Consequently, the projection $\ka^\PL(x)$ of the cocycle $\ka(x)$ (associated to $x \in X(Z)$) to $\pi^\PL(X)$ is simply a $\Gm$-equivariant homomorphism
\[
\pi^\un_1(Z) \to \pi^\PL(X).
\]

\begin{sslm}
\label{diamond}
The functions $f_0$ and $f_1, f_{10}, f_{100}, \dots$ on $\pi_1^\un(X)$ factor through $\pi^\PL(X)$ and form a set of homogeneous coordinates on the latter.
\end{sslm}

\begin{proof}
After forgetting the $\pi_1^\un(Z)$ action, $\nN(X)$ is just the free pronilpotent Lie algebra on generators $e_0, e_1$, and the ideal associated to $N$ is generated by $e_1$. So this is purely formal.
\end{proof}

\noindent
We will denote the function $f_0$ on $\pi^\PL(X)$ by $\logu$ and the function $f_{10 \cdots 0}$ with $(n-1)$ zeroes by $\Liu_n$.

\ssegment{table}{}
Let $\nN = \nN(Z) = \Lie \pi_1^\un(Z)$. We define the \emph{Goncharov quotient} by
\[
\nN^G(Z): = \nN / [\nN_{\le -2}, \nN_{\le -2}].
\]
We also consider the associated quotient $\pi^G(Z)$ of $\pi_1^\un(Z)$ and the associated Hopf subalgebra
\[
A^G(Z) \subset A(Z).
\]

\begin{sslm}
\label{shmats1}
Every $\Gm$-equivariant homomorphism
\[
\pi_1^\un(Z) \to \pi^\PL(X)
\]
factors through $\pi^G(Z)$.
\end{sslm}

\begin{proof}
Any degree zero graded homomorphism 
\[
\nN(Z) \to \nN^\PL(X)
\]
must send $\nN(Z)_{\le -2}$ to $\nN^\PL(X)_{\le -2}$. But
\[
[\nN^\PL(X)_{\le -2},\nN^\PL(X)_{\le -2}] =0.
\qedhere
\]
\end{proof}

\begin{ssprop}
\label{shmats2}
The unipotent motivic polylogarithmic values $\Liu_i(x)$ ($x\in X(Z)$) belong to $A^G(Z)$. 
\end{ssprop}

\begin{proof}
Returning to \ref{unfun}(*), we saw in lemma \ref{diamond} that the function $f_{10 \cdots 0}$ factors through $\pi^\PL(X)$, and we saw in lemma \ref{shmats1} that the cocycle $\ka^\PL(x)$ factors through $\pi^G(Z)$.
\end{proof}

\begin{ssprop}
The map
\[
Z^1(\pi_1^\un(Z), \pi^\PL(X))^\Gm
 \from
   Z^1(\pi^G(Z), \pi^\PL(X))^\Gm
\]
induced by pulling back a $\Gm$-equivariant cocycle along the projection
\[
\pi_1(Z) \surj \pi^G(Z)
\]
is bijective. In particular (given that we're working with the polylogarithmic quotient of $\pi_1^\un(X)$), replacing the full mixed Tate Galois group $\pi_1^\un(Z)$ by its Goncharov quotient will have no further effect on the resulting loci $X(\Zp)_n$.
\end{ssprop}

\begin{proof}
Direct consequence of lemma \ref{shmats1}.
\end{proof}

We consider the following strengthening of the depth-1 part of Goncharov's depth filtration conjecture \cite{GonGal}. For any prime $q$, let 
\[
Z_{> q} =
 \Spec \ZZ  \setminus 
 \{ 
 \mbox{primes} \le q
 \}.
\]
For any graded algebra $A$, we let $A_{[\le n]}$ denote the subalgebra generated by elements of degree $\le n$.

\begin{ssconj}[Integral depth-1 conjecture]
\label{depth}
For every prime $q_s$ and every $n \in \NN$, there exists a $q_M \ge q_s$ such that $A^G_{[\le n]}(Z_{>q_M})$ is generated as a $\QQ$-algebra by the elements $\logu(q')$ for $q'$ prime $\le q_M$, the elements $\zeu(n)$ for $n$ odd $\ge 3$ and the elements $\Liu_i(a)$ for $n \ge i \ge 2$ and $a \in X(Z_{>q_M})$.
\end{ssconj}

\ssegment{}{}
As an example, $A^G(\Spec \ZZ)$ is spanned by $\zeu(n)$ for $n$ odd $\ge 3$, so the integral depth-1 conjecture holds for $\Spec \ZZ$.

\ssegment{rk}{Remark}
Let us put this conjecture in context with a brief preview of things to come. In section \ref{tbasis} below we construct an algorithm that takes a natural number $n$ and a prime number $q_s$ as input, and that upon halting outputs two primes $q_s \le q_M < p$ and a family of points $a_{i,j} \in X(Z_{>q_M})$, such that
\begin{enumerate}
\item
if the algorithm halts then $\{\Liu_i(a_{i,j})\}_{i,j}$ forms an algebra basis of $A^G(Z_{>q_M})$ in half-weights $\le n$, and
\item
if conjecture \ref{depth} holds, then the algorithm halts.
\end{enumerate}
We emphasize that while the halting is conditional, the validity of the output is unconditional. In particular, our algorithm may be used to verify the conjecture experimentally up to any given bound on the weight. Our choice to work with open subschemes of $Z$ of the form $Z_{>q_M}$ (instead of searching through all open subschemes, as we do in \cite{mtmueii}) means that the scope of our search, as we attempt to construct a basis, is smaller. It has the disadvantage, however, that we may have to increase the auxiliary prime $p$.

\ssegment{rki}{Remark}
Let us further explain the role played by the scheme $Z_{>q_M}$ in our point-counting algorithm. For general $Z$, $A^G(Z)$ may not be generated by polylogarithms $\Liu_i(a)$ with $a\in X(Z)$. For example, if $Z = \Spec \ZZ[1/q]$ for $q$ a prime $\neq 2$ then
\[
A^G(\ZZ) \neq A^G(\ZZ[1/q])
\]
(and $A^G(\ZZ[1/q])$ is not even generated over $A^G(\ZZ)$ by $\logu(q)$) but
\[
X(Z) = \emptyset.
\]
Moreover, we have no direct way of deciding if a given rational linear combination of $\Liu_n(a)$'s ($a\in X(\QQ)$) is unramified over $Z$ (i.e. is contained in $A(Z)$). Our method, present already in a more general (but less precise) form in \cite{mtmueii}, is rather indirect. We first \textit{taper} $Z$ down to a $Z_{>q_M}$. We then construct a polylogarithmic basis of $A^G(Z_{>q_M})$ (up to a given weight). If we insist that our basis be compatible with the extension spaces, we may then generate an associated shuffle basis. We then construct a change-of-basis matrix which relates our polylogarithmic basis to our shuffle basis. Finally, we use our shuffle basis to identify the subspace 
$
A^G(Z) \subset A^G(Z_{>q_M}).
$\footnote{
However, there's a caveat, which is why there's a further remark, which should be thought of as being in smaller print. 
}

\ssegment{para}{Further to remark \ref{rki}}
Actually, since our change of basis matrix is only a $p$-adic approximation, we can only identify the subspace $A^G(Z)$ inside $A^G(Z_{>q_M})$ up to given $p$-adic precision. This means that, for all we know, the polylogarithmic basis we construct may only be a linearly independent set of the right size, $p$-adically close to $A^G(Z)$. This issue may be dealt with roughly as follows. We make sure that our basis elements are sufficiently spread out (modifying if necessary) to ensure that their projection onto $A^G(Z)$ remains linearly independent. We refer to the resulting basis as our ``abstract basis'' since its definition is not constructive. We subsequently carry out all computations for both the abstract and the (``concrete'') polylogarithmic basis, keeping track of the accumulated error. 

Issues of this sort were already treated quite carefully in \cite{mtmueii}, which is consequently littered with double-tildes. Here we limit ourselves to mentioning these issues only in passing, since they tend to cloud the exposition and obscure the essential features of our construction. Rather, we feel that these issues are best relegated to a future computational article in which the algorithm presented here is given in a sort of paragraph-style pseudo code, divorced entirely from its theoretical backdrop. \textit{That} article, we hope, will come along with Sage code and a wealth of numerical results. In the meantime, in working out small-scale examples with a human touch (as we did in \cite{CorwinDCI}), these issues can usually be circumvented by replacing $p$-adic approximations of relations between values of polylogarithms by actual known relations found in the literature about  polylogarithms.\footnote{
But here again there is a caveat, which is why there's yet another remark, in an even smaller font.
}

\ssegment{glue}{Further to remark \ref{para}}
Actually, this particular issue was dealt with differently in \cite{mtmueii}. There, we completely avoided constructing a basis of $A(Z)$. Instead, after finding an open subscheme $Z^o \subset Z$ (which plays the role played by our $Z_{>q_M}$) and a basis of $A(Z^o)$, we modified the universal evaluation map to allow coefficients in the larger $A(Z^o)$. This meant that the difference between $A(Z)$ and $A(Z^o)$ was dealt with within the \textit{geometric algorithm}. However, as we noticed during our work on \cite{CorwinDCI}, this made our algorithm quite inefficient in cases of interest. For instance, it meant that the geometric algorithm for the case $Z = \Spec \ZZ[1/q]$ must grow with $q$. 

A different issue, which comes up only for larger number fields, was, however, dealt with in a manner similar to that indicated in remark \ref{para}. When a basis for the extension spaces is not known, we are unable to construct one algorithmically. Instead, we construct a linearly independent set $p$-adically close to the extension space in question and proceed as indicated above.

\segment{watch}{Application to integral points}
Let $I_\m{BC} = \per^\sharp$ denote the map
\[
\Spec \Qp \to \pi_1^\un(Z)
\]
induced by the period map (since it is essentially given by \textit{Besser-Coleman Integration}). Let $\pi_{\ge -n}^\PL(X)$ denote the quotient by the $n$th step of the descending central series. As explained in the proof of Proposition 2 of Kim \cite{kimi},
 the functor from $\QQ$-algebras to sets
\[
R \mapsto 
Z^1(\pi^G(Z)_R, \pi_{\ge -n}^\PL(X)_R)^\Gm=
 \Hom^\Gm(\pi^G(Z)_R, \pi_{\ge -n}^\PL(X)_R)
\]
of $\Gm$-equivariant cocycles, is represented by an affine space 
\[
\mathbf{Z}^1(\pi^G(Z), \pi_{\ge -n}^\PL(X))^\Gm
\]
over $\QQ$.\footnote{
A particularly concrete construction of the isomorphism to affine space is given in \cite[Corollary 3.11]{CorwinDCI}
 and indicated in remark \ref{phi} below.
} Let $\ev^G$ denote the universal evaluation map
\[
 \pi^G(Z) \times
\mathbf{Z}^1
\big(\pi^G(Z), \pi_{\ge -n}^\PL(X)\big)^\Gm 
 \to
 \pi^G(Z) \times  \pi_{\ge -n}^\PL(X) 
\]
\[
\ev^G(\ga,\phi) = (\ga, \phi(\ga)).
\]

\ssegment{cousin}{}
Upon taking $\Qp$-points (as $\QQ$-schemes), the evaluation map fits into a commuting square
\small
\[
\xymatrix{
X(Z) \ar[r] \ar[d]_-{\ka}
&
X(\Zp) \ar[d]^-{\be}
\\
\Big(
\pi^G(Z) \times
\mathbf{Z}^1
\big(\pi^G(Z), \pi_{\ge -n}^\PL(X)\big)^\Gm
\Big) (\Qp) 
\ar[r]_-{\ev^G}
&
\Big(
\pi^G(Z) \times  \pi_{\ge -n}^\PL(X) 
\Big)(\Qp)
}
\]
sending 
\[
\xymatrix{
x \ar@{|->}[d]_-{\ka}
&&
y \ar@{|->}[d]^-{\be}
\\
\Big(  I_\m{BC}, \ka(x) \Big) 
\ar@{|->}[r]_-{\ev^G}
&
\left(
I_\m{BC},
{\begin{pmatrix}
\log^p(x) \\
 \Li_1^p(x)
 \\ \Li_2^p(x) 
\\ \vdots
\end{pmatrix}}
\right)
& 
\left(
I_\m{BC},
{\begin{pmatrix}
\log^p(y) \\
 \Li_1^p(y)
 \\ \Li_2^p(y) 
\\ \vdots
\end{pmatrix}}
\right)
}
\]
which constitutes our computable cousin of ``Kim's cutter''.\footnote{
We recall that the commutativity is a rather long story which, at least in one possible approach, starts with Olsson's nonabelian $p$-adic Hodge theory; see \cite{mtmue}.
}

\ssegment{sym}{}
For us, the \emph{polylogarithmic Chabauty-Kim locus} $X(\Zp)_n$ is a locally analytic (or ``Coleman analytic'') subspace of $X(\Zp)$. To define it, we take the scheme theoretic image of $\ev^G$, we pull back along $\be$, and finally we symmetrize with respect to the $S_3$ action. This last step means that we close the set of generators under the two operations $F(z) \mapsto$ 
\begin{align*}
\tag{*}
F(1-z),
&&
\mbox{and}
&&
F \left( \frac{1}{z} \right).
\end{align*}
We also consider $X(Z)$ as a locally analytic subspace of $X(\Zp)$ with reduced structure.\footnote{
In \cite{CorwinDCI} we use the notation $X(\Zp)_n^{S_3}$ to distinguish this locus from the \textit{un}-symmetrized version considered previously. 
} 

\begin{ssconj}[Convergence of polylogarithmic loci]
\label{kim}
Let $Z$ be an open subscheme of $\Spec \ZZ$ and let $p$ be a closed point of $Z$. Then for $n$ sufficiently large we have an equality of locally analytic subspaces of $X(\Zp)$
\[
X(Z) = X(\Zp)_n.
\]
(In particular, for such $n$, the points of $X(Z)$ are not double roots of the $n$th Chabauty-Kim ideal.)
\end{ssconj}

\ssegment{wave}{Remark}
In segment \ref{localg} below, we complete the construction of an algorithm which upon halting computes the locus $X(\Zp)_n$ to given precision; this is our \emph{loci algorithm}. This algorithm proceeds in several steps. In one of these, which we refer to as the \emph{geometric algorithm} (\S\ref{geomalg}), we compute the scheme theoretic image of the map $\ev^G$ with respect to coordinates induced by a shuffle basis for $A(Z)$. In a second step, we use the change of basis matrix constructed in the \emph{change of basis algorithm} (\S\ref{cob}) to convert the resulting equations into equations in coordinates associated to the polylogarithmic basis constructed in our \emph{basis algorithm} (\S\ref{tbasis}). The remaining steps amount to a straightforward application of \textit{Lip service} \cite{Lip}. Pulling back along $\be$ merely means replacing unipotent motivic polylogarithms by $p$-adic polylogarithms. The symmetrization was spelled out in \ref{sym}(*). Finally, the algorithm of loc. cit. allows us to obtain local power series expansions to given $p$-adic and geometric precision.

\ssegment{pence}{Remark}
In terms of the \emph{loci algorithm} our \emph{point counting algorithm} from \cite{mtmueii} remains unchanged. We review its construction in segment \ref{pcalg} below, referring back to \cite{mtmueii} for details. In terms of the conjectures above and the algorithm below, our point-counting theorem (which may equally be called a point-\textit{finding} theorem) is as follows. If $Z$ is an open subscheme of $\Spec \ZZ$ in which the largest prime \emph{excluded} is $q_M$, we refer to $Z_{> q_M}$ as \emph{the tapered scheme associated to $Z$}. 

\begin{ssthm}[Point-counting]\label{pointcountingthm}
Let $Z$ be an open subscheme of $\Spec \ZZ$. 
\begin{enumerate}
\item
If the point-counting algorithm halts for the input $Z$, then its output is equal to $X(Z)$.
\item
Assume the conjectured nonvanishing of the $p$-adic zeta values $\ze^p(k)$ ($k$ odd $\ge 3$). If the integral depth-1 conjecture (\ref{depth}) holds for the tapered scheme associated to $Z$, and if convergence of polylogarithmic loci (\ref{kim}) holds for $Z$, then the point-counting algorithm halts for the input $Z$.
\end{enumerate}
\end{ssthm}

\noindent
Unlike in \cite{mtmueii}, here we do not separate the proof of the \textit{point-counting theorem} from the construction of the \textit{point-counting algorithm}. Rather, we set ourselves the task of computing the data to be computed, and explain in down to earth terms, how we go about computing algorithmically. Thus, we consider the theorem to be proved as soon as the algorithm has been constructed. This task occupies the remainder of the article.

\section{Setup}

\segment{}{}
We continue to work with an open subscheme $Z$ of $\Spec \ZZ$ and a prime $p \in \ZZ$. Recall that $\pi_1^\un(Z)$ denotes the unipotent part of the fundamental group of mixed Tate motives unramified over $Z$ and $\nN(Z)$ denotes its Lie algebra, which has a natural grading --- we call the graded degree of a homogeneous element its \emph{half-weight}. Recall that $\nN^G(Z)$ denotes the Goncharov quotient of $\nN(Z)$ and that 
\[
\pi^G(Z) = \exp \nN^G(Z)
\]
denotes the associated quotient of $\pi_1^\un(Z)$. Recall that 
 $A(Z) = \Oo(\pi_1^\un(Z))$ denotes the graded Hopf algebra of functions on $\pi_1^\un(Z)$, and that $A^G(Z)$ denotes the subalgebra associated to $\pi^G(Z)$.
 
 Let
\[
d_i = \dim_\QQ \nN^G(Z)_i.
\]
Let $D_n^G(Z)$ denote the image of the product map
\[
\bigoplus_{i+j = n, \; i, j \ge 1} A^G_i \otimes A^G_j \to A^G_n.
\]
The Lie coalgebra $L^G := (\nN^G)^\lor$ is equal to the quotient
\[
L^G = A^G_{>0}/D^G = A^G_{>0} / (A_{>0}^G)^2
\]
of the augmentation ideal by its square.

Recall that $\Uu(Z) = A(Z)^\lor$ denotes the completed universal enveloping algebra of $\pi_1^\un(Z)$ and, adding the decoration `$G$' as usual, $\Uu^G(Z) = A^G(Z)^\lor$ denotes the completed universal enveloping algebra of $\pi^G(Z)$.

\begin{sprop}
\label{anyset}
Any set of homogeneous elements of $A^G_{>0}$ which maps to a basis of $L^G$ forms an algebra basis for $A^G$. 
\end{sprop}

\begin{proof}
After forgetting the counit and comultiplication, $A^G$ has the structure of a graded free $\QQ$-algebra 
\[
A^G = \QQ[S]
\]
with $S = \bigcup_{i=1}^\infty$ and $S_i$ finite for each $i$. This is simply because $A^G$ is the coordinate ring of a prounipotent group with $\Gm$-action such that the graded pieces of the abelianization are finite dimensional. Let $I = A^G_{>0}$ denote the ideal of $\QQ[S]$ of positively graded elements.

Let
\[
S' = \bigcup_{i=1}^\infty S'_i
\]
be a set of homogeneous elements of $A^G_{>0}$ which maps to a basis of $L^G = I/I^2$ and let $I'$ be the ideal of positively graded elements in $\QQ[S']$. Then the induced map of $\QQ$-algebras
\[
\phi: \QQ[S'] \to \QQ[S]
\]
preserves the grading and induces an isomorphism $I'/I'^2 \to I/I^2$. Hence by \cite[3.1.1]{mtmueii}, $\phi$ is an isomorphism.
\end{proof}

\segment{bat}{Remark}
Recall from \S\ref{mtm}
 that the kernel $E_n = E_n(Z)$ of the reduced coproduct on $A(Z)_n$ is canonically isomorphic to the space 
\[
\Ext^1_Z(\QQ(0), \QQ(n))
\]
of extensions in mixed Tate motives over $Z$. Similarly, the kernel $E^G_n = E^G_n(Z)$ of the reduced coproduct on $A^G_n(Z)$ is equal to a space of extensions in the full subcategory of the category of mixed Tate motives consisting of objects whose associated representation factors through the Goncharov quotient. We will refer to such objects as ``Goncharov motives''.

It follows directly from the definition however, that we have an equality of spaces of extensions $E_n^G = E_n$; in the case at hand, both are spanned by logarithms and by the motivic zeta elements $\zeu(n)$ for $n$ odd. Consequently, the category of Goncharov motives must have nontrivial higher extension groups. This complicates the structure of the Hopf algebra $A^G$. As mentioned in the introduction, instead of analyzing its structure, we will work inside of $A$.

Let $q_M$ denote a prime sufficiently large compared to $q_s$ and $n$, and let $Z_{> q_M}$ denote the subscheme of $\Spec \ZZ$ obtained by removing all primes $\le q_M$ as above. After constructing a polylogarithmic algebra basis $\Bb^G$ for 
$A^G(Z_{>q_M})$ 
(which includes the zeta elements) in half-weights $\le n$, we will extend our basis arbitrarily to an algebra basis $\Bb$ of $A(Z_{>q_M})$. We will then define the generators $\si_{r}$, $r$ odd $\ge 3$, to be dual to the zeta elements relative to the given choice of basis. We will then have according to proposition 3.2.3 of \cite{mtmueii},
\[
\Uu(Z_{>q_M}) = 
\QQ\langle \langle 
\{ \tau_q \}_{q \le q_M}, 
\{\si_r\}_{r \mbox{ odd } \ge 3}
\rangle \rangle
\]
but with the $\si_r$ well defined only in the quotient
\[
\Uu \surj \Uu^G.
\]

\segment{bronze}{}
Returning to $X = \thrpl$, we recall that $\pi_1^\un(X)$ denotes the unipotent fundamental group at the tangent vector $\vec{1_0}$, and that $\pi^\PL(X)$ denotes its polylogarithmic quotient. We recall that the polylogarithmic quotient has canonical coordinates, which we denote by $\logu, \Liu_1, \Liu_2, \dots$, so that
\[
\Oo(\pi^\PL(X)) = \QQ[\logu, \Liu_1, \Liu_2, \dots]
\]
with $\logu$ in degree $1$ and $\Liu_i$ in degree $i$. Recall that in \S\ref{unfun} we associated to a $Z$-valued base-point $a$ of $X$ a 1-cocycle 
\[
\ka(a): \pi_1^\un(Z) \to \pi^\PL(X)
\]
and defined the (unipotent) motivic $n$-logarithm of $a$ by
\[
\Liu_n(a) := \ka(a)^\sharp(\Liu_n).
\]
More generally, if $R$ is a $\QQ$-algebra and 
\[
c: \pi_1^\un(Z)_R \to \pi^\PL(X)_R
\]
is a family of cocycles parametrized by $\Spec R$, we set
\[
\Liu_n(c):= c^\sharp(\Liu_n),
\]
an element of $A(Z) \otimes R$. Similarly, we set
\[
\logu(c) := c^\sharp(\logu).
\]

\sProposition{c}{
We denote the reduced coproduct by $\Delta'$. We have
\[
\Delta'\Liu_n = \sum_{i=1}^{n-1} \frac{(\logu)^{ i}}{i!}
\otimes \Liu_{n-i}.
\]
}

\begin{proof}
In view of the formula
\[
(\Liu_{e_0})^{ m} = m! \Liu_{(e_0)^m},
\]
this is just the deconcatenation coproduct of shuffle coordinates on a free prounipotent group.
\end{proof}

\sCorollary{d}{
A similar formula holds for $\Delta' \Liu_n(c)$ for any cocycle $c$:
\[
\Delta'\Liu_n(c) = \sum_{i=1}^{n-1} \frac{(\logu(c))^{ i}}{i!}
\otimes \Liu_{n-i}(c),
\]
 as well as for $\Delta'\Liu_n(z)$ for any $z \in X(Z)$.
}

\begin{proof}
Since $\pi_1^\un(Z)$ acts trivially on $\pi^\PL(X)$, a cocycle
\[
c: \pi_1^\un(Z) \to \pi^\PL(X)
\]
is simply a group homomorphism. This means $c^\sharp$ preserves coproducts. Hence,
\begin{align*}
\Delta'\Liu_n(c)
	&= \Delta' c^\sharp \Liu_n \\
	&= (c^\sharp \otimes c^\sharp)(\Delta' \Liu_n) \\
	&= 
	(c^\sharp \otimes c^\sharp)
	\left(
	\sum_{i=1}^{n-1} \frac{(\logu)^{ i}}{i!} \otimes \Liu_{n-i}
	\right) \\
	&= \sum_{i=1}^{n-1} \frac{(\logu(c))^{ i}}{i!}
	\otimes \Liu_{n-i}(c).
\end{align*}
\end{proof}

\segment{gonco}{Remark}
Corollary \ref{d} may be upgraded in (at least) four different ways: (1) by replacing $c$ by a cocycle $\pi^\un_1(Z) \to \pi^\un_1(X,x)$ which is no longer a homomorphism, (2) by further puncturing $X$ at a finite number of rational points, (3) by replacing $\Liu_n$ by the function $\pi^\un_1(X,x) \to \AA^1_\QQ$ associated to an arbitrary word in a basis of $H_1^\dR(X)$, and (4) by replacing $\QQ$ by a general number field. The result (at least when $c$ comes from a rational point $y$) is Theorem 1.2 of Goncharov \cite{GonGal}, which has come to be known as the \emph{Goncharov coproduct formula}. As we mentioned in the introduction, one advantage of working with $n$-logarithms is that we can make do with the much simpler formula of Corollary \ref{d}.

\segment{ceqsetup}{}
Fix arbitrarily a set 
\[
\Si = \bigcup_{i \le -1} \Si_i
\]
of homogeneous free generators for $\nN(Z)$. Any word $w$ in the generators is naturally an element of the completed universal enveloping algebra $\Uu(Z)$. We denote the natural pairing 
\[
A(Z) \otimes_\QQ \Uu(Z) \to \QQ
\]
by $\langle - , - \rangle$, and the same after base-change to any $\QQ$-algebra $R$.

\begin{sprop}
\label{ceq}
Continuing with the situation and the notation of segments \ref{bronze}, \ref{ceqsetup}, fix natural numbers $ 0 \le r <n$ and elements $\tau_1, \dots, \tau_r \in \Si_{-1}$, and $\si \in \Si_{r-n}$.
 We then have
\[
\langle
\Liu_n (c),\si \tau_1 \cdots \tau_r
\rangle
 =
\langle
\Li^\uU_r (c),\si
\rangle
\langle \logu (c), \tau_1\rangle
\cdots 
\langle \logu (c), \tau_r \rangle
\]
and all other values $\langle \Li^\uU_n (c),w\rangle$ ($w$ a word in $\Si$) vanish.
\end{sprop}

\begin{proof}
This formula appeared in a letter written by Francis Brown to I. Dan-Cohen and is proved in \cite[Proposition 3.10]{CorwinDCI}. 
\end{proof}

\segment{}{Definition}
Let $k$ be a field with an absolute value. We say that vectors
\[
v_1, \dots, v_d \in k^n
\]
are \emph{$\ep$-linearly independent} if there exists a $d \times d$ minor whose determinant has absolute value $> \ep$.

\segment{}{Proposition}
Let $k$ be a field with an absolute value $| \cdot |$ and $\tilde v_1, \dots , \tilde v_d$ vectors in $k^n$ which are \emph{$\ep$-linearly independent}.   Then there exists a number $\ep'>0$, algorithmically computable from the data $(\ep,\tilde v_1, \dots , \tilde v_d)$, which goes to $0$ as $\ep \to 0$, and such that for any family $v_1, \dots, v_d$ of vectors in $k^n$, if 
\[
|v_i -\tilde v_i| < \ep'
\]
for each $i = 1, \dots, d$, then the vectors $v_1, \dots , v_d$ are linearly independent. 

\begin{proof}
The proof reduces to the case $d = n$, to which we now restrict attention. The algorithmic computability will remain implicit. Given  $v = (v_1, \dots, v_d)$, $\tilde v = (\tilde v_1, \dots, \tilde v_d)$ as in the proposition, we can bound 
\[
| \det v - \det \tilde v |
\]
 by a positive number $\delta$ which depends only on $\ep'$ and on $\tilde v_1, \dots, \tilde v_d$, and decreases monotone to $0$ as $\ep' \to 0$. We outline the construction. For each $i = 1, \dots, n$, we let $\Delta_i$ denote the norm of the linear functional 
 \[
\det(\tilde v_1, \dots, \tilde v_{i-1}, \bullet, \tilde v_{i+1}, \dots, \tilde v_n).
\]
Then there is a positive number $\de_i$ (with the same properties as above) such that the norm of the linear functional 
\[
\det(v_1, \dots, v_{i-1}, \bullet, \tilde v_{i+1}, \dots, \tilde v_n)
\]
is bounded by $\Delta_i + \de_i$ whenever 
\[
|v_j - \tilde v_j | < \ep'
\]
for $j = 1, \dots, i-1$. Setting $\Delta$ equal to the maximum among $\Delta_1 + \de_1, \dots, \Delta_d+\de_d$, we find that
\begin{align*}
| \det v &- \det \tilde v | \\
	& \le \sum_{i=1}^d  \big| 
	\det(v_1, \dots, v_i, \tilde v_{i+1}, \dots, \tilde v_d ) -
	\det(v_1, \dots,  v_{i-1}, \tilde v_{i}, \dots, \tilde v_d )
	\big| 
	\\
	&=  \sum_{i=1}^d  \big| 
	\det(v_1, \dots, v_{i-1}, v_i - 
	\tilde v_i, \tilde v_{i+1}, \dots, \tilde v_d )
	\big|
	\\
	& \le \sum_{i=1}^d  \big| 
	\det(v_1, \dots, v_{i-1}, \bullet, \tilde v_{i+1}, \dots, \tilde v_d )
	\big| 
	\cdot |v_i - \tilde v_i |
	\\
	& < \sum_{i=1}^d (\Delta_i+\de_i) \cdot \epsilon'
	\\
	& \le d \cdot  \Delta \cdot \ep'
\end{align*}
whenever 
$
|v_j - \tilde v_j | < \ep'
$
for $j = 1, \dots, d$. Thus, we set $\delta := d \cdot  \Delta \cdot \ep'$ to complete our outline of the construction of $\delta$. The monotonicity of $\delta$ as a function of $\ep'$ ensures that $\ep' \mapsto \de$ can be inverted. 

Turning to the proof of the proposition, we assume that $\det v > \ep$ and we find that 
\begin{align*}
| \det v | & \ge | \det \tilde v | - | \det v - \det \tilde v |
> \ep - \frac{\ep}{2} =  \frac{\ep}{2} > 0 
\end{align*}
whenever
\[
|v_j - \tilde v_j | < \ep'(\de = \ep/2)
\]
for $j = 1, \dots, d$, as required. 
\end{proof}

\section{Basis for receding $Z$}
\label{tbasis}

We construct an algorithm which takes as input a prime $q_s$
and a natural number $n$, and outputs two primes $q_s \le q_M<p$ and a doubly indexed family
\[
\{ a_{i,j} \}_{\underset{ 1 \le j \le d_i}{2 \le i \le n}}.
\]
For $i$ odd, we will set $a_{i,1} = \vec{-1_1}$, so that
\[
\Liu_i(a_{i,1}) = \zeu(i).
\]
The remaining elements $a_{i,j}$ are $Z_{>q_M}$-points of $X$. We name this algorithm \emph{basis for receding $Z$}. We first announce the meaning of the output in a proposition. 
If $A$ is a graded algebra, we let $A_{[\le n]}$ denote the subalgebra generated by elements of graded degrees $\le n$. If $A$ is a polynomial algebra, we refer to a set of free generators as an \emph{algebra basis}.

\begin{sprop}
Suppose the algorithm \emph{basis for receding $Z$} (mentioned above and constructed in segments \ref{half}--\ref{basisalg} below) halts on the input $(q_s, n)$ and, upon halting, outputs the data 
\[
\left( q_M, p, 
\{ a_{i,j} \}_{\underset{ 1 \le j \le d_i}{2 \le i \le n}}
\right).
\]
Then the unipotent logarithms
$
\logu q
$
for $q$ prime $\le q_M$, together with the unipotent polylogarithms and zeta elements 
\[
\Liu_i(a_{i,j})
\]
for $2 \le i \le n$ and $1 \le j \le d_i$, form an algebra basis for $A^G_{[\le n]}(Z_{>q_M})$. In particular, the integral depth-one conjecture (\ref{depth}) holds for $q_s$ in half-weights $\le n$. Additionally, the $p$-adic zeta values $\ze^p(m)$ for $m$ odd $\in [3, n]$ are nonzero.
\end{sprop}

\segment{half}{Subalgorithm}
We begin by constructing a subalgorithm that will be applied recursively within the main algorithm. The input consists of 
\begin{itemize}
\item
two primes $q_M<p$,
\item
an $\ep \in p^{-\NN}$,
\item
a height-bound $b\in \NN$; we denote by $X(Z_{> q_M})_b \subset X(Z_{> q_M})$ the set of elements of height $\le b$,
\item
a family of elements $a_{i,j} \in X(Z_{> q_M})_b$ such that the associated family of motivic polylogarithms 
\[
 L_{i,j} = \Liu_i(a_{i,j})
 \]
forms an algebra basis
\[
\Bb^G =
\{ 
L_{i,j} 
 \}
 _{i \le n}
\]
 for $A^G_{[\le n]}(Z_{>q_M})$; we set $L_{1,j} = \logu(q_j)$,  and $L_{i,1} = \zeu(i)$ for $i$ odd $\ge 3$.
\end{itemize}
The output is a function (in the form of a list or a ``dictionary'')
 which assigns to any $a \in X(Z_{>q_M})_b$, and any $m \le n$, an expansion of $\Liu_m(a)$ in the monomial vector-space basis associated to the algebra basis $\Bb^G$ to precision $\ep$. The construction is recursive in $m$.
 
 \ssegment{cobalt}{}
In the presence of the basis $\Bb^G$, the $\QQ$-algebra $A^G_{[\le n]}(Z_{>q_M})$ may be identified as a vector space with a space of vectors with entries in the field $\QQ$ of rational numbers equipped with the $p$-adic absolute value, and the computations that follow are carried out there. As a matter of notation, we let 
\[
\Aa = \{\Aa_{m,j} \}
\]
denote the monomial vector space basis associated to the algebra basis $\Bb^G$, numbered so that
\[
\Aa_{m,j} = L_{m,j}
\]
for $j = 1, \dots, \dim \nN^G_{-m}$, and $\Aa_{m,j}$ is a shuffle monomial in $L_{m', j'}$ with $m' < m$ for $j > \dim \nN^G_{-m}$. For the base case of our recursive construction, we have
\[
\Liu_1(a) = -\logu(1-a) \in 
A^G_1 \cong \QQ^{\{ q_1, \dots, q_M \}}
, 
\]
which we may expand in the logarithms $\logu q_j$ by decomposing $1-a$ as a product of primes.

\ssegment{stove}{}
 Assume we've expanded the polylogarithmic values $\Liu_{<m}(a)$ in our basis for $A^G_{[<m]}(Z_{>q_M})$ up to precision $\ep$. Our basis gives us in particular a basis
 \[
 \Aa' = \Aa_1 \times \Aa_{m-1} \cup 
 \Aa_2 \times \Aa_{m-1} \cup 
 \cdots
 \cup \Aa_{m-1} \times \Aa_1
 \]
  for the direct sum of tensor products 
\[
A^G_1 \otimes A^G_{m-1} + 
A^G_2 \otimes A^G_{m-2} + \cdots
+ A^G_{m-1} \otimes A^G_1,
\]
and allows us to identify the latter with a space $\QQ^N$ of vectors.

Assume $m$ odd (the case $m$ even is simpler). Recall from \ref{mtmoverview}($*^\lor$) that in this case, the reduced coproduct $\Delta'$ on $A^G_m$ is injective modulo the motivic zeta value $\zeu(m)$. Because of our imperfect approximations, $\Delta' \Liu_m(a)$ may not quite be in the linear span of the images $\Delta' \Aa_{m,j}$ of the basis elements $\Aa_{m,j}\in A^G_m$. We may nevertheless \emph{project} $\Delta' \Liu_m(a)$ onto the subspace spanned by the $\Delta' \Aa_{m,j}$ (relative to the basis $\Aa'$) and compute the coefficients:
\[
\Delta' \Liu_m(a) = \sum_{j=2}^{\dim A^G_m} c_{j} 
\Delta' \Aa_{m,j}.
\]
To do so, we expand each vector $\Delta' \Aa_{m,j}$ in the basis $\Aa'$ to precision $\ep$, and we expand the new vector $\Delta' \Liu_m(a)$ in the basis $\Aa'$ to precision $\ep$ as well. We then set
\[
c_j := 
\langle 
\Delta' \Liu_m(a), 
\Delta' \Aa_{m,j} 
\rangle
\]
where the inner product is the standard inner product on $\QQ^N$.

This gives us all coefficients except for the coefficient $c_1$ of $\zeu(m)$. To determine the latter, we use the period map as follows. Letting $\Li_m^\ep(a)$, $\Aa^\ep_{m,j}$ denote $\ep$-approximations of the $p$-adic periods of $\Liu_m(a)$, $\Aa_{m,j}$ produced by the algorithm of Besser--de Jeu \cite{Lip}, we set 
\[
c_1 := \frac
{ \Li^\ep_m(a) - \sum_{j=2}^{\dim A^G_m} c_{j} \Aa^\ep_{m,j}}
{\ze^\ep(m)}
\]
(decreasing $\ep$ if needed so as to achieve $|\ze^\ep(m)| > \ep$). We then have the expansion
\[
\Liu_m(a) \underset \ep \sim \sum c_j \Aa_{m,j}
\]
we hoped for. This completes the construction of the subalgorithm. 

\segment{basisalg}{Main algorithm}
We now construct the main algorithm of this section. Recall that we are given as input a prime $q_s$ and a natural number $n$. Our primary goal is to construct a basis for $A^G_{[\le n]}(Z_{>q_M})$ using $p$-adic approximations, where $q_M$ is a prime  $ \ge q_s$ and $p$ is a prime $> q_M$. Along the way we will search for potential basis elements among the points 
\[
X(Z_{>q_M})_b \subset X(Z_{>q_M})
\]
of height $\le b$. Since there may not be enough of these points, we will enthusiastically increase $b$ while reluctantly considering the possibility of increasing $q_M$, and hence $p$. Our secondary goal is to ensure that our $p$-adic approximations are good enough to witness the nonvanishing (known for $p$ regular, conjectured in general) of the $p$-adic zeta values $\ze^p(m)$ for $m$ odd $\in [3,n]$.

Let $I$ denote the set of quadruples $(q_M, p, b, \ep)$ with $p > q_M \ge q_s$ primes, $b \in \NN$, and $\ep \in p^{-\NN}$. Let $ J \subset I$ be a subset with the following properties. (1) The projection of $J$ onto the coordinate plane $(b, \ep)$ defines $\ep$ as a decreasing function of $b$. (2) There exists an increasing function $b \mapsto p_b$ from the natural numbers to the set of primes, such that for fixed $(b,\ep)$, the fiber of $J$ above $(b,\ep)$ is equal to the set of all pairs of primes $q_M < p \le p_b$ in which $p$ is the \textit{next} prime after $q_M$. We \emph{arbitrarily} impose an ordering on the set $J$.

\ssegment{basrk2}{Remark}
Psychologically, we may imagine $q_M$, $p$ and $b$ to be \emph{increasing} while $\ep$ \emph{decreases}. However, it is important that after decreasing $\ep$, we also \emph{decrease} $q_M$ and $p$. Thus, as the algorithm proceeds, we occasionally revisit past primes in order to give them a second chance. 

\ssegment{basrk3}{Remark}
Given $(q_M, p, b, \ep) \in J$ then, the algorithm attempts to verify the nonvanishing of the $p$-adic zeta values and to build a basis for
\[
A^G_{[\le n]}(Z_{>q_M})
\]
using the points of $X(Z_{> q_M})_b$ and using $p$-adic approximations of precision $\ep$. This may fail for several reasons. One reason is that an $\ep$-approximation of one of the $p$-adic zeta values $\ze^p(m)$ may equal $0$. A second reason is that, having potentially succeeded in constructing a partial basis, there may not be another linearly independent polylogarithm available among the points of $X(Z_{>q_M})_b$ and this may be because $b$ is too small or because the entire set of $Z_{>q_M}$-points $X(Z_{>q_M})$ is too small. Finally, even if an appropriate choice of next basis element \emph{can} be found inside $X(Z_{>q_M})_b$, our $\ep$-approximations may be too coarse to see the linear independence. 

\ssegment{besisalgrec}{}
We now assume, in preparation for the recursive step, that we've reached a data point $(q_M,p,b,\ep) \in J$. Assume further that we have constructed an algebra basis for $A^G_{[< m]}(Z_{>q_M})$ for some $m \le n$ consisting of motivic logarithms, motivic zeta values, and motivic polylogarithms $\Liu_i(a_{i,j})$ with $a_{i,j} \in X(Z_{>q_M})_b$. Assume further that we've constructed a partial algebra basis in weight $m$ given by unipotent $m$-logarithms
\[
L_1 = \Liu_m(a_{m,1}), \dots, L_r= \Liu_m(a_{m,r})
\]
($a_{m,j} \in X(Z_{>q_M})_b$)
which are linearly independent modulo the space $D_n^G$ of decomposables.  We choose arbitrarily a point $a \in X(Z_{>q_M})_b$ and an $\ep \in p^{-\NN}$ and consider adding $L := \Liu_m(a)$ to our basis. We use the subalgorithm of segment \ref{half} above to expand $\Delta'$ of the  decomposables in weight $m$, the $\Delta'L_i$, as well as our new candidate $\Delta'L = \Delta'\Liu_m(a)$ with $p$-adic precision $\ep$ in our polylogarithmic basis
\[
\bigcup_{i+j = m, \; i,j \ge 1} \Aa_i \times \Aa_j
\]
 for the space
\[
\bigoplus_{i+j = m, \; i,j \ge 1} A^G_i\otimes A^G_j
\]
and check the result for $\ep$-linear independence (ignoring $L_1$ if $m$ is odd). 
\begin{itemize}
\item
If the result is negative, we go on to the next quadruple $(q_M,p,b,\ep) \in J$. If $p$ has changed, we  verify the nonvanishing of an $\ep$-approximation $\ze^\ep(m)$ of $\ze^p(m)$ for $m$ odd $\in [3,n]$, decreasing $\ep$ as needed.
\item
 If the result is positive, we set $a_{m,r+1}$ equal to $a$ and continue the process. 
  \end{itemize}
 We halt when we reach $r = \dim \nN^G_m$ for each $m \le n$. This completes the construction.

\begin{sprop}
Assume $p$ regular (or nonvanishing of the $p$-adic zeta values $\ze^p(m)$ for $m$ odd $\ge 3$).
Assume the integral depth-1 conjecture (\ref{depth}) holds for $q_M$ in half weights $\le n$. Then the algorithm of segment \ref{basisalg} halts. 
\end{sprop}

\section{Change of basis}
\label{cob}

\segment{freegen}{}
We continue to work with the scheme $Z_{>q_M}$ produced in \S\ref{tbasis} and we drop the repeated argument `$(Z_{>q_M})$' throughout this section. Recall that $A_{[\le n]} \subset A$ denotes the subalgebra generated by elements of half-weight $\le n$, and similarly for $A_{[\le n]}^G$. Given our polylogarithmic algebra-basis
\[
\Bb^G_{\le n} = \bigcup_{i = 1}^{n} \Bb_i^G
\]
 for $A^G_{[\le n]}$, there exists an algebra basis
$
\Bb_{\le n} = \bigcup_{i = 1}^{n} \Bb_i
$
for $A_{[\le n]}$ which extends $\Bb^G_{\le n}$, as well as associated free generators
\[
\tau_1, \dots, \tau_M, \; \si_3, \si_5, \si_7, \dots
\]
of the Lie algebra $\nN = \nN(Z_{> q_M})$ (in half-weights bounded below by $-n$), where we identify the latter with the set of Lie-like elements in the completed universal enveloping algebra (c.f. remark \ref{bat} above). We detail the construction of this algebra basis, in notation chosen to accord with \cite[\S3.1]{mtmueii}. Let $\Pp_{\le n}^G \subset \Bb_{\le n}^G$ denote the subset obtained by removing the extension classes (i.e. the logarithms and zeta elements). In half-weight $1$ we set
\[
\Pp_1 = \Pp_1^G = \emptyset.
\]
In half-weight $i \ge 2$ we extend the set
$
\Pp_i^G
$
(by choosing arbitrary linearly independent elements of $A_i \subset A$)
 to a linearly independent subset $\Pp_i$ of $A_i$ which spans a linear complement to the subspace
\[
E_i \oplus D_i \subset A_i
\]
spanned by extensions and decomposables. Set
\begin{align*}
\Ee_i := 
\left\{
\begin{array}{ll}
\{\logu q_1, \dots, \logu q_M \} 
& i = 1
\\
\zeu(i) & i >1 \mbox{ odd}
\\
\emptyset & i >1 \mbox{ even},
\end{array}
\right.
&&
\Ee_{\le n} := \bigcup_{i = 1}^n \Ee_i,
&&
\Pp_{\le n} := \bigcup_{i=1}^n \Pp_i,
\end{align*}
and
\[
\Bb_{\le n} := \Ee_{\le n} \cup \Pp_{\le n}.
\]
Let $\tau_i$ be the element of $\Uu_{-1}$ (half-weight $-1$ part of the completed universal enveloping algebra) dual to $\logu q_i$ relative to the basis 
\[
\Bb_1 = \{\logu q_1, \dots, \logu q_M \}
\]
of $A_1$. For $q = q_i$ a prime $\le q_M$, we sometimes write $\tau_q$ in place of $\tau_i$.

Let $\Aa_r$ denote the set of monomials of half-weight $r$ in the set $\Bb_{\le n} \subset A$. For $r$ odd ($3\le r \le n$) let $\si_r$ be the element of $\Uu_{-r}$ dual to $\zeu(r)$ relative to the vector space basis $\Aa_r$ of $A_r$. Then according to propositions 3.2.2 and 3.2.3 of \cite{mtmueii},
\[
A_{[\le n]} = \QQ[ \Bb_{\le n}].
\]
as $\QQ$-algebras, and the $\tau_q$, $\si_r$ form free generators as hoped.\footnote{
The resulting algebra basis $\Bb_{\le n}$ is a mixture of concrete polylogarithmic elements of $A^G$ which we have constructed algorithmically on the one hand, with abstract elements of $A$ on the other hand, whose construction does not intervene in the algorithm. If we were to separate our construction of the algorithm from our verification that its output has the desired meaning, then these last elements would serve as a mere book-keeping device in the construction.
}

\segment{}{}
If $w$ is a word in the generators $\tau_q$, $\si_r$, and $\Aa$ is an element of the vector-space basis for $A^G$ generated by the polylogarithmic algebra basis (i.e. a monomial in $\Bb^G_{\le n}$), then the value
\[
\langle \Aa, w \rangle \in \QQ
\]
is independent of the choice of basis for $A$ beyond the polylogarithmic basis constructed for $A^G$.

We now construct an algorithm which takes as input a polylogarithmic basis, a word $w$ in the generators $\tau_q$, $\si_r$, an element $\Aa$ of the vector space basis generated by the algebra basis, and an $\ep$, and computes $\langle \Aa, w \rangle$ to precision $\ep$. 
The construction is a d\'evissage in three steps. An example is worked out for instance in \S7.6.3 of \cite{mtmue} as well as in \cite{CorwinDCI}.

\ssegment{}{}
If $\Aa = \Aa' \Aa''$ is a product of two or more algebra-basis elements, we use the relationship  
\[
\langle \Aa' \Aa'', w \rangle 
 = \langle \Aa' \otimes \Aa'', \mu (w) \rangle
\]
between the shuffle product on $A$ and the coproduct $\mu$ on its completed universal enveloping algebra $\Uu$ repeatedly to  reduce to the case that $\Aa = L$ is itself an algebra-basis element. 

\ssegment{}{}
The values $\langle L, w \rangle$ for $L$ a logarithm, an $n$-logarithm, or a zeta value obey the following rules. We have
\[
\langle \logu a, \tau_q \rangle = v_q(a)
\]
(the $q$-adic valuation of $a$)
and all other values 
$\langle \logu a, w\rangle $ vanish. We have
\[
\langle \Liu_1 a, \tau_q \rangle = 
\langle -\logu(1-a), \tau_q \rangle = -v_q(1-a).
\]
By proposition \ref{ceq}, we have 
\[
\langle \Liu_n a, \tau_1\cdots \tau_n \rangle = 
\langle \Liu_1 a, \tau_1 \rangle
\langle\logu a, \tau_2 \rangle
 \cdots 
 \langle \logu a, \tau_n \rangle,
\]
and
\[
\langle \Liu_n a, \si_r \tau_1 \cdots \tau_s \rangle
 =
\langle \Liu_r a, \si_r \rangle 
\langle \logu a, \tau_1 \rangle
\cdots 
\langle \logu a, \tau_s \rangle
\]
($r+s = n$), and all other values $\langle \Liu_n a, w \rangle$  vanish. Finally, by definition 
\[
\langle \zeu(n) , \si_n \rangle = 1
\]
and all other values $\langle \zeu(n), w \rangle$ vanish. Using these formulas, we reduce to the computation of the values $\langle \Liu_r (a), \si_r \rangle$, noting, however, that $\Liu_r(a)$ may not be an algebra basis element. 

\ssegment{}{}
We use the method of \S\ref{half} to expand $\Liu_r (a)$ in our polylogarithmic basis in half-weight $r$ to precision $\ep$. We have thus reduced to the case that $\Aa = L$ is again an algebra-basis element, while $w = \si_r$ is a one-letter word. Finally, by our very definition of $\si_r$, we have 
\[
\langle L, \si_r \rangle =
\left\{
\begin{array}{ll}
    1    \mbox{ if }  L = \zeu(r),  \mbox{ and}  \\
    0  \mbox{ otherwise.}
\end{array}
\right.
\]
This completes the construction of the algorithm.

\segment{}{Remark}
Given $w \in \Uu_{-k}$ a word in the generators $\tau_p$, $\si_r$ (with $k \le n$), we let $f_w \in A_k$ denote the dual element relative to the basis consisting of such words. In terms of the resulting shuffle basis, the above computations can be rewritten as follows:
\[
\logu q = f_{\tau_q},
\]
\[
\zeu(n) = f_{\si_n},
\]
and
\[
\Liu_i(a) = 
\sum 
\langle \Liu_r a , \si_r \rangle
 v_{q_1}(a)\cdots v_{q_s}(a) 
f_{\si\tau_{q_1}\cdots \tau_{q_s}} 
+
\sum
 v_{q_0}(a)\cdots v_{q_i}(a) 
f_{\tau_{q_0}\cdots \tau_{q_i}} 
\]
($a = a_{i,j}$, $r+s = i$). 

\section{Basis for $Z \subset \Spec \ZZ$ arbitrary}
\label{arb}

\segment{}{}
We now consider $Z \subset Z_{>q_M}$ arbitrary. At this point we have a polylogarithmic basis $\{\Aa_{i,j}\}$ for $A^G_{[\le n]}(Z_{>q_M})$, a shuffle basis $\{f_w\}$ for all of $A(Z_{>q_M})$ (both in bounded weights $\le n$), and a matrix $M$ expanding the former in the latter to precision $\ep$, which we think of as the matrix associated to the inclusion
\[
A_{\le n}^G(Z_{>q_M}) \subset A_{\le n}(Z_{>q_M})
\]
(where the subscript $\le n$ refers to the vector subspace of elements in graded degrees $\le n$) relative to the polylogarithmic basis on the source and the shuffle basis on the target. Relative to the shuffle basis, $A_{\le n}(Z) \subset A_{\le n}(Z_{> q_M})$ is the hyperplane spanned by $f_w$ with $w$ not involving the generators $\tau_q$ for primes $q \in Z$. Pulling back via $M$, we obtain a system of linear equations. We may then construct a basis for the space of solutions by basic methods of linear algebra. The result is a vector space basis $\Aa^G$ of $A_{\le n}^G(Z)$.

\section{Geometric algorithm}
\label{geomalg}

\segment{geom1}{}
The \emph{geometric algorithm} takes as input a finite set
\[
\Si_{-1} = \{\tau_1, \dots, \tau_s\}
\]
and a natural number $n$, and outputs a finite set
\[
\{F_1^a, \dots, F_N^a\}
\]
 of elements of the polynomial algebra 
\[
\QQ[\{f_\la \}_\la, \logu, \Liu_1, \Liu_2, \dots, \Liu_n]
\]
where $\la$ ranges over the set of Lyndon words in the (suitably ordered) set
\[
\Si = \bigcup_{i = 1}^n \Si_{-i}
\] 
where $\Si_{-i}$ contains one element $\si_i$ for $i$ odd $\ge 3$ and no elements for $i$ even. (The superscript `$a$' stands for \emph{abstract}.)

\segment{this1}{}
This algorithm is independent of the previous algorithms, and its halting is unconditional. We first explain the meaning of its output. Let $\pi^\un(\Si)$ be the free prounipotent group on the set $\Si$ with $\Gm$-action induced by placing $\Si_i$ in graded degree $i$. Let $\nN = \nN(\Si)$ be its Lie algebra --- the free pronilpotent Lie algebra on the set $\Si$. Let $\nN^G(\Si)$ be the quotient
\[
\nN^G(\Si) := \nN / [ \nN_{\le -2}, \nN_{\le -2}],
\]
let $\pi^G(\Si)$ be the associated quotient of $\pi^\un(\Si)$ and let $A^G(\Si) = \Oo(\pi^G(\Si))$ be the associated Hopf algebra. There is a canonical isomorphism
\[
A(\Si) := \Oo(\pi_1^\un(\Si)) = \QQ[\{f_\la\}_\la], 
\]
hence an inclusion 
\[
A^G(\Si) \subset A(\Si) = \QQ[\{f_\la\}_\la].
\]
Let $\nN^\PL_{\ge -n}$ be the graded Lie algebra
\[
\nN^\PL_{\ge -n} = \QQ(1) \ltimes \prod_{i=1}^n \QQ(i)
\]
(with $\QQ(i)$ placed in graded degree $-i$) and let $\pi^\PL_{\ge -n}$ denote the associated unipotent $\QQ$-group. Let
\begin{align*}
\logu = f_0,
&&
\Liu_1 = f_1,
&&
\Liu_2 = f_{10}, 
&&
\dots
&&
\Liu_n = f_{10 \cdots 0}
\end{align*}
denote shuffle coordinates on $\pi^\PL_{\ge -n}$ associated to its presentation as a quotient 
\[
\pi^\un(e_0, e_1) \surj \pi^\PL_{\ge -n}
\]
of the free prounipotent group on two generators $e_0$, $e_1$. Endow $\pi^\PL_{\ge -n}$ with the trivial $\pi^G(\Si)$-action. Let $\ev^G_n(\Si)$ denote the evaluation map
\[
\pi^G(\Si) \times Z^1(\pi^G(\Si), \pi^\PL_{\ge -n})^\Gm
\to
\pi^G(\Si) \times \pi^\PL_{\ge -n}.
\]
\[
\ev^G_n(\Si)(\ga, c) = (\ga, c(\ga)).
\]
Then the functions $F^a_1, \dots, F^a_N$ produced by the algorithm are contained in
\[
A^G(\Si)[\logu, \Liu_1, \dots, \Liu_n]
\]
where they generate the ideal associated to the scheme theoretic image of $\ev^G_n(\Si)$.

\segment{place1}{}
We now construct the algorithm. As a corollary of proposition \ref{ceq} above, we find in \cite[Corollary 3.11]{CorwinDCI} that the \emph{full} cocycle evaluation map
\[
\ev_n(\Si): \pi^\un(\Si) \times
Z^1
\big(\pi^\un(\Si), \pi^\PL_{\ge -n} \big)^\Gm 
 \to
 \pi^\un(\Si) \times  \pi^\PL_{\ge -n} 
\]
(i.e. without passing to the Goncharov quotient) is given in coordinates by the map of finite type polynomial algebras over $\QQ$
\[
\QQ[ \{f_\la \}_\la, \{\Phi^\rho_\Lambda\}_{\wt (\rho) = \wt (\Lambda) } ]
\from
\QQ[\{f_\la \}_\la, \logu, \Liu_1, \Liu_2, \dots, \Liu_n],
\]
(where $\la$ ranges over Lyndon words in the generators $\tau$, $\si$ of $\pi_1^\un(Z)$, $\Lambda$ ranges over the set of polylogarithmic words in $e_0$, $e_1$
 of weight $\le n$, and $\rho$ ranges over the set of generators $\tau$, $\si$ of $\pi^\un_1(Z)$) given by
\[
\sum_{\tau \in \Si_{-1}} f_\tau \Phi^\tau_0 
\mapsfrom
\logu
\]
and
\[
\sum
_{
\begin{matrix}
\tau_1, \dots, \tau_r \in \Si_{-1} \\
\si \in \Si_{-s} \\
r+s = n \\
1 \le s \le n
\end{matrix}
}
f_{\tau_1 \cdots \tau_r \si}
 \Phi^{\tau_1}_0 \cdots \Phi^{\tau_r}_0 
\Phi^\si_{\underbrace{0 \cdots 01}_{s}}
\mapsfrom
\Liu_n.
\]
Since 
\[
Z^1
\big(\pi^G(\Si), \pi^\PL \big)^\Gm
=
Z^1
\big(\pi^\un(\Si), \pi^\PL \big)^\Gm,
\]
the scheme theoretic image of $\ev_n^G(\Si)$ is the same as the scheme theoretic image of the composite 
\[
 \pi^\un(\Si) \times
Z^1
\big(\pi^\un(\Si), \pi^\PL_{\ge -n} \big)^\Gm 
 \to
 \pi^\un(\Si) \times  \pi^\PL_{\ge -n} 
 \to
\pi^G(\Si) \times  \pi^\PL_{\ge -n}.
\]
In terms of coordinate rings, this means restricting the map $\ev_n(\Si)^\sharp$ to the subalgebra
\[
 A^G(\Si)[ \logu, \Liu_1, \Liu_2, \dots, \Liu_n].
\]
By proposition \ref{anyset} (which, in turn, is a direct application of  \cite[3.1.1]{mtmueii}), a basis for the latter may be constructed out of the basis $\{f_\la\}$ for $\pi^\un(\Si)$ by elementary linear algebra. Subsequently, a set of generators for the kernel of
\[
\QQ[ \{f_\la \}_\la, \{\Phi^\rho_\Lambda\}_{\wt (\rho) = \wt (\Lambda) } ]
\from
A^G(\Si)[\logu, \Liu_1, \Liu_2, \dots, \Liu_n]
\]
may be constructed by standard methods of elimination theory. This completes our construction of the geometric algorithm.

\segment{phi}{Remark}
If $\Cc$ is the universal cocycle
\[
\pi^\un_1(Z) \times
Z^1
\big(\pi_1^\un(Z), \pi^\PL(X)\big)^\Gm 
 \to
   \pi^\PL(X) \times
Z^1
\big(\pi_1^\un(Z), \pi^\PL(X)\big)^\Gm,
\]
then 
\[
\Phi^\rho_{\underbrace{0 \cdots 01}_{n}} (\Cc)
=
\langle \Liu_n(\Cc), \rho \rangle
\]
(this equality takes place inside the coordinate ring of $Z^1$, or, in terms of our shuffle basis, inside $\QQ[\{\Phi^\rho_\Lambda\}]$). In other words, $\Phi^\rho_{0\dots01}$ corresponds to the function on cocycles $\langle \Liu_n(?), \rho \rangle$.

\section{Loci algorithm}

In terms of the basis algorithms, the change of basis algorithm, and the geometric algorithm, our \emph{loci algorithm} is similar to the algorithm of \cite[\S4.2]{mtmueii}; we repeat the construction, making adjustments as needed. 

\segment{}{}
The loci algorithm takes as input an open subscheme $Z \subset \Spec \ZZ$, a natural number $n$ and an $\ep$. As output, it returns a prime $p \in Z$, a polylogarithmic algebra basis $\Bb^G_{\le n}$ of the polynomial ring $A^G_{[\le n]}(Z)$ and a family $\{F_i\}_i$ of elements of the polynomial ring
\[
\QQ[\Bb^G_{\le n}, \logu, \Liu_1, \dots, \Liu_n].
\]

\segment{}{}
Before constructing the algorithm, we explain the meaning of its output upon halting. There's an obvious homomorphism 
\[
\QQ[\Bb^G_{\le n}, \logu, \Liu_1, \dots, \Liu_n]
\to
\operatorname{Col}(X(\Zp))
\]
to the ring of Coleman functions. Let $F_i^p$ denote the image of $F_i$. We symmetrize the family $\{F_i^p\}_{i \in I}$ with respect to the $S_3$ action as indicated in segment \ref{sym} to obtain a bigger family $\{F_j^p\}_{j \in J}$. Then the family $\{F_j^p\}_{j \in J}$ is within $\ep$ of a set of generators for the ideal of $\operatorname{Col}(X(\Zp))$ which defines the (symmetrized) polylogarithmic Chabauty-Kim locus $X(\Zp)_n$.

\segment{localg}{}
The algorithm is constructed as follows. We run the \emph{geometric algorithm} (\S\ref{geomalg}) on the set of symbols
\[
\Si_{-1} = \{q_1, \dots, q_s\}
\]
with $s$ equal to the number of primes excluded from $Z$. This gives us a family $\{F^a_i\}$ of elements of the polynomial $\QQ$-algebra 
\[
\QQ[\{f_\la \}_\la, \logu, \Liu_1, \Liu_2, \dots, \Liu_n].
\]
Thus, the \emph{coefficients} of $F^a_i$ are elements of the vector space $\QQ[\{f_w\}]$  with basis indexed by words $w$ in the set
$\{
\tau_1, \dots, \tau_s, \; \si_{-3}, \si_{-5}, \dots
\}$.
(In fact, by their construction, the coefficients will belong to the subspace corresponding to the Goncharov quotient.)

We run the algorithm \emph{basis for receding $Z$} (\S \ref{tbasis}) on the input $(q_s, n)$ to obtain a pair of primes $q_s \le q_M <p$ and a polylogarithmic algebra basis
\[
\tag{*}
_{>q_M}\Bb^G_{\le n}
\]
(hence also an associated monomial vector space basis $_{>q_M}\Aa^G_{\le n}$) of
\[
A_{[\le n]}^G(Z_{>q_M}).
\]
We run the \emph{change of basis algorithm} (\S\ref{cob}) on the polylogarithmic algebra basis (*) and we run the \emph{basis algorithm for arbitrary $Z$} (\S\ref{arb}) on the further input $Z$. By elementary linear algebra, we obtain
\begin{enumerate}
\item
a polylogarithmic algebra basis $\Bb^G_{\le n}$ for $A^G_{[\le n]}(Z)$, and
\item
 for each element $\Aa$ of the associated vector space basis ${\Aa^G_{\le n}},$
an associated vector $M(\Aa)$ in the vector space $\QQ[\{f_w\}]$.
\end{enumerate}
We now expand the coefficients of each $F_i^a$ in the vectors $M(\Aa)$ to obtain a family $\{F_i\}$ of elements of 
the polynomial ring
\[
\QQ[\Bb^G_{\le n}, \logu, \Liu_1, \dots, \Liu_n]
\]
as hoped. This completes the construction of the algorithm.

\section{Point counting algorithm}
\label{pcalg}

\segment{ab}{}
Our \emph{root criterion algorithm} from \cite[\S5]{mtmueii} combines standard methods of Newton polygons together with a growth estimate obtained by Besser--de Jeu \cite{Lip} to decide whether the number of zeroes of a $p$-adic power series in a given ball is zero or one, given a sufficiently close approximation. We do not repeat it here.

\segment{ac}{}
Our \emph{point-counting algorithm} from \cite[\S7]{mtmueii} remains unchanged; we nevertheless do repeat it for the reader's convenience, while avoiding the double tildes of loc. cit. as indicated in remark \ref{para}. This algorithm takes as input an open subscheme $Z$ of $\Spec \ZZ$ and proceeds by running two processes simultaneously. One process is simply a naive search for points of $X(Z)$. This produces a gradually increasing subset $X(Z)_n \subset X(Z)$.

The other process locates an appropriate prime $p \in Z$ and computes the loci $X(\Zp)_n$ to given precision. It then verifies if it is possible, with the given level of precision, to declare an equality 
\[
X(Z)_n = X(\Zp)_n.
\]

\segment{aa}{}
We begin by running our basis algorithm. In addition to the polylogarithmic basis of $A^G_{[\le n]}(Z)$ which will remain fixed throughout the remainder of the construction, this also gives us the auxiliary prime $p$.

Having done so, our algorithm searches through the set of triples $(n, N, \ep)$, $n, N \in \NN$, $\ep$ in a countable subset of $\RR_{> 0}$ with accumulation point $0$.  After each attempt, we increase $n$ and $N$ and decrease $\ep$. To each such triple, our algorithm assigns a set $X(Z)_n$ of points of $X(Z)$ and a boolean. The boolean output will be constructed in segments \ref{ba}--\ref{bd}. If the boolean output is {\it True}, then we output $X(Z)_n$. If the boolean output is {\it False}, then we continue the search. To produce the set $X(Z)_n$, we simply search for points up to a suitable hight-bound depending on $n$ which goes to infinity with $n$. The remainder of the construction concerns the boolean output. 

\segment{ba}{}
We partition $X(\Zp)$ into $\ep$-balls, decreasing $\ep$ as needed to ensure that each ball contains at most one element of the set $X(Z)_n$ (our, potentially incomplete, list of integral points). We run our loci algorithm to produce a family $\{F_i\}_i$ of polylogarithmic functions on $X(\Zp)$.

\segment{bb}{}
We now focus our attention on an $\ep$-ball $B$ containing a rational representative $y \in B$. Using \textit{Lip service} \cite{Lip}, we expand each polylogarithmic function $F_i$ to arithmetic precision $\ep$ and geometric precision $e^{-N}$ about $y$. In a technical step explained in the proof of theorem 7.2.1 of \cite{mtmueii}, we must check that all nonzero coefficients are larger than $\ep$, returning \textit{False} if not.

\segment{bc}{}
Let $b$ be the number of points ($0$ or $1$) in $X(Z)_n \cap B$. We run the root-criterion algorithm on the ball $B$, on the precision-levels $N$ and $\ep$, and on each of the functions $F_i$, to verify if $B$ contains no more than $b$ roots.

\segment{bd}{}
We repeat steps \ref{bb}-\ref{bc} in each ball. This completes the construction of the algorithm.

\bibliography{PolGonII_Refs}
\bibliographystyle{alphanum}

\vfill

\noindent 
\Small\textsc{
David Corwin \\
Department of mathematics \\
970 Evans Hall \#3840 \\
University of California \\
Berkeley, CA 94720-3840
}
\\
{Email Address:} \texttt{corwind@alum.mit.edu}

\bigskip 

\noindent
\Small\textsc{
Ishai Dan-Cohen \\ 
Department of mathematics \\ 
Ben-Gurion University of the Negev\\ 
Be'er Sheva, Israel}
\\ 
{Email address:} \texttt{ishaidc@gmail.com}

\end{document}